\documentclass[11pt]{article}

\newif\ifarxiv
\arxivtrue

\usepackage{ifthen} 
\usepackage{xifthen} 
\usepackage{etoolbox}

\ifthenelse{\isundefined{\arxivtrue}}{\newif\ifarxiv}{}
\ifthenelse{\isundefined{\aplastrue}}{\newif\ifaplas}{}

\usepackage{amsmath}
\usepackage{amssymb}
\makeatletter
\@ifclassloaded{llncs}{\newcommand{\qedhere}{}}{\usepackage{amsthm}}
\makeatother

\ifarxiv
  \usepackage[
    paperwidth=15.5cm,paperheight=23.5cm,footskip=23pt,  
    textwidth=360pt,textheight=595.8pt   
  ]{geometry}
\fi

\usepackage{mathtools}

\ifarxiv
  \usepackage[T1]{fontenc}
  \usepackage{libertine}
  \usepackage[libertine]{newtxmath}
  \useosf 
  \usepackage[scaled=0.93]{zi4}  
  \usepackage{microtype}
\fi

\usepackage{xcolor}
\definecolor{darkgreen}{rgb}{0,0.45,0}
\definecolor{darkred}{rgb}{0.75,0,0}
\definecolor{darkblue}{rgb}{0,0,0.6}
\usepackage[colorlinks,citecolor=darkgreen,linkcolor=darkred,urlcolor=darkblue]{hyperref}
\usepackage{breakurl}
\usepackage[inline,shortlabels]{enumitem}
\setlist{noitemsep}

\usepackage{aliascnt} 
\newcommand{\cnewtheorem}[2]{%
  \newaliascnt{#1}{theorem}
  \newtheorem{#1}[#1]{#2}
  \aliascntresetthe{#1}
}

\usepackage[nameinlink,capitalize]{cleveref} 
\makeatletter
\@ifclassloaded{llncs}
  {\spnewtheorem*{remark*}{Remark}{\itshape}{\rmfamily}
   \newenvironment{customremark}[1][]{\begin{remark*}}{\end{remark*}}}
  {
  \theoremstyle{plain}
  \newtheorem{theorem}{Theorem}
  \cnewtheorem{proposition}{Proposition}
  \cnewtheorem{lemma}{Lemma}
  \cnewtheorem{fact}{Fact}
  \cnewtheorem{corollary}{Corollary}
  \cnewtheorem{conjecture}{Conjecture}

  \theoremstyle{definition}
  \cnewtheorem{definition}{Definition}
  \newtheorem*{remark*}{Remark}
  \cnewtheorem{remark}{Remark}
  \cnewtheorem{example}{Example}

  \newtheorem{innercustomremark}{Remark}
  \newcommand{\custremend}{}
  \newenvironment{customremark}[1][]
    {\@ifmtarg{#1}%
      {\begin{remark*} \renewcommand{\custremend}{\end{remark*}}}
      {\renewcommand\theinnercustomremark{#1}\begin{innercustomremark} \renewcommand{\custremend}{\end{innercustomremark}}}}
    {\custremend}
}

\makeatother

\usepackage{mathpartir}
\usepackage{graphicx}

\ifarxiv
\makeatletter
  
   \makeatletter
\fi

\usepackage{titlecaps}
\newcommand{\titlecaseifaplas}[1]{\ifaplas\titlecap{#1}\else#1\fi}
\newcommand{\autocapsection}[1]{\section{\titlecaseifaplas{#1}}}
\newcommand{\autocapsubsection}[1]{\subsection{\titlecaseifaplas{#1}}}

\Addlcwords{a an the}
\Addlcwords{and or but}
\Addlcwords{with as of for}

\usepackage{tikz}
\usepackage{tikz-cd}

\usetikzlibrary{calc}
\usetikzlibrary{fit}
\usetikzlibrary{shapes}
\usetikzlibrary{arrows}
\usetikzlibrary{decorations.markings}
\usetikzlibrary{decorations.pathmorphing}
\usetikzlibrary{positioning}
\usetikzlibrary{intersections}

\ifthenelse{\isundefined{\arxivtrue}}{\newif\ifarxiv}{}
\ifthenelse{\isundefined{\aplastrue}}{\newif\ifaplas}{}

\newcommand{\defemph}[1]{\emph{#1}}

\makeatletter
\g@addto@macro\@floatboxreset\centering
\makeatother

\usepackage[colorinlistoftodos,prependcaption,textsize=tiny]{todonotes}

\newif\ifshownotes
\shownotestrue
\AtBeginDocument{\ifshownotes \GenericWarning{}{Private notes visible.} \fi}

\newcommand{\hidenotes}{\shownotesfalse}

\newcommand{\PLL}[1]{\ifshownotes \todo[color=blue!30,author=PLL]{#1} \else \ignorespaces \fi}

\newcommand{\PLLtodo}[1]{\GenericWarning{}{PLL: #1}\PLL{#1}}

\makeatletter

\newcommand\freefootnote[1]{%
  \begingroup

  \renewcommand\@makefntext[1]{%
  \parindent 0em%
  \noindent
  #1}

  \footnotetext{#1}%
  \endgroup
}
\makeatother

\newcounter{diagram}  

\newenvironment{diagram}[1][]{%
    \begin{equation}%
    \setcounter{diagram}{\theequation}
    \addtocounter{diagram}{-1}
    \refstepcounter{diagram}
}{%
    \end{equation}%
}
\crefname{diagram}{diagram}{diagrams}
\crefname{diagram}{Diagram}{Diagrams}
\creflabelformat{diagram}{#2(#1)#3}

\DeclareRobustCommand{\abbrevcrefs}{%
\crefname{theorem}{Th.\!}{Th.\!}%
\crefname{definition}{Df.\!}{Df.\!}%
\crefname{lemma}{Lem.\!}{Lem.\!}%
\crefname{proposition}{Pr.\!}{Pr.\!}%
\crefname{corollary}{Cor.\!}{Cor.\!}%
\renewcommand{\crefpairconjunction}{,\nobreakspace}%
}
\DeclareRobustCommand{\dref}[1]{{\abbrevcrefs\textrm{\cref{#1}}}}

\tikzset{
  commutative diagrams/diagrams={row sep=large},
  commutative diagrams/arrow style=tikz,
  arrlabel/.style={font=\footnotesize},
  arrlabelsmall/.style={font=\tiny},
  arrmidlabel/.style={auto=false,fill=white,font=\footnotesize} 
}

\tikzset{cd-style/.style={commutative diagrams/every diagram}}
\tikzset{cd-arrow-style/.style={commutative diagrams/.cd, every arrow, every label}}

\pgfdeclarearrow{name=fibtip,means={Triangle[open,angle=60:4.5pt]}}
\pgfdeclarearrow{name=fibbtip,%
  means={Triangle[sep=-0.4pt,open,angle=60:4.5pt,fill=white] Triangle[open,angle=60:4.5pt]}}

\pgfarrowsdeclarealias{c}{c}{left hook}{right hook}
\pgfarrowsdeclarealias{c'}{c'}{right hook}{left hook}

\tikzset{drpb/.style={commutative diagrams/.cd, dr, phantom, "\lrcorner", very near start}}
\tikzset{
  adjsymbol/.style={commutative diagrams/.cd, phantom,"\scriptstyle \bot"{sloped}}, 
  adjshiftl/.style={commutative diagrams/.cd, shift right=-2*#1},
  adjshiftl/.default=1,
  adjshiftr/.style={commutative diagrams/.cd, shift right=2*#1},
  adjshiftr/.default=1,
  adjbendl/.style={commutative diagrams/.cd, bend left=20},
  adjbendr/.style={commutative diagrams/.cd, bend right=20}
}

\tikzset{
  dfixanch/.style={commutative diagrams/.cd, start anchor=south, end anchor=north},
  ufixanch/.style={commutative diagrams/.cd, start anchor=north, end anchor=south},
  ladjd/.style={commutative diagrams/.cd, dfixanch, adjbendr, adjshiftr={1.8-0.15*#1}},
  ladjd/.default=0,
  radjd/.style={commutative diagrams/.cd, dfixanch, adjbendl, adjshiftl={1.8-0.15*#1}},
  radjd/.default=0,
}

\tikzset{fib/.code={\pgfsetarrowsend{fibtip}}}
\tikzset{fibb/.code={\pgfsetarrowsend{fibbtip}}}
\tikzset{inj/.code={\pgfsetarrowsstart{c}}}
\tikzset{weq/.style={"{\simeq}"}}
\tikzset{weq'/.style={"{\simeq}"'}}
\tikzset{iso'/.style={"\cong"'}}
\tikzset{sup/.style={label={#1}{auto=left,pos=1}}}  
\tikzset{sub/.style={label={#1}{auto=right,pos=1}}}

\newcommand{\generalto}[2]{ \mathrel{\mkern-1mu
  \tikz[baseline={([yshift=-0.55ex]a.south)}]{%
    \node[minimum width=1.5em,align=center,inner xsep=0.5ex,inner ysep=0.15ex] (a) {$\scriptstyle #2$};
    \draw[#1] (a.south west) -- (a.south east);}
 \mkern-1mu}}

\newcommand{\generalfrom}[2]{ \mathrel{\mkern-1mu
  \tikz[baseline={([yshift=-0.55ex]a.south)}]{%
    \node[minimum width=1.5em,align=center,inner xsep=0.5ex,inner ysep=0.15ex] (a) {$\scriptstyle #2$};
    \draw[#1] (a.south east) -- (a.south west);}
  \mkern-1mu}}

\renewcommand{\to}[1][]{ \generalto{->}{#1} }
\newcommand{\tto}[1][]{ \generalto{double,double equal sign distance,-implies}{#1} }

\newcommand{\fibto}[1][]{ \generalto{fib}{#1} }
\newcommand{\injto}[1][]{ \generalto{->,inj}{#1} }
\makeatletter
\newcommand{\equivto}[1][]{\@ifmtarg{#1}{\generalto{->}{\sim}}{\NotYetDefined}}
\newcommand{\equivfrom}[1][]{\@ifmtarg{#1}{\generalfrom{->}{\sim}}{\NotYetDefined}}
\makeatother

\newcommand{\downfibbarrow}{\mathrel{\tikz[baseline=2pt,y=2.2ex,x=2.5pt]{
    \useasboundingbox (-1,0) rectangle (1,1) ;
    \draw[-{Triangle[sep=-0.4pt,open,angle=70:3.5pt,fill=white] Triangle[open,angle=70:3.5pt]}]
        (0,1) -- (0,0) ;
      }}}

\tikzcdset{
    boxedcd/.style={every matrix/.append style={draw,inner ysep=5pt}},
}

\newcommand{\C}{\ensuremath{\mathcal{C}}}
\newcommand{\D}{\mathcal{D}}

\newcommand{\T}{\ensuremath{\mathcal{T}}}

\newcommand{\DM}{\mathcal D} 

\newcommand{\catname}[1]{\ensuremath{\mathbf{#1}}}
\newcommand{\catsscript}[1]{{\mathrm{#1}}}

\newcommand{\Cat}{\catname{Cat}}
\newcommand{\Clan}{\catname{Clan}}
\newcommand{\CC}{\catname{CompCat}}

\newcommand{\CompCat}{\CC}
\newcommand{\CwA}{\catname{CwA}}
\newcommand{\CwF}{\catname{CwF}}
\newcommand{\CxlCat}{\catname{CxlCat}}
\newcommand{\DMC}{\catname{DMC}}

\newcommand{\FinSet}{\catname{FinSet}}
\newcommand{\Lex}{\catname{Lex}}

\newcommand{\NatMod}{\catname{NatMod}}
\providecommand{\Top}{}\renewcommand{\Top}{\catname{Top}}
\newcommand{\sDMC}{\catname{sDMC}}
\newcommand{\Set}{\catname{Set}}

\newcommand{\full}{\catsscript{full}}
 
\newcommand{\subcat}{\catsscript{sub}} 
\newcommand{\replete}{\catsscript{repl}} 
\newcommand{\compclosed}{\catsscript{compcl}}

\renewcommand{\split}{\catsscript{spl}} 
\newcommand{\discrete}{\catsscript{disc}} 
\newcommand{\pointed}{\diamond} 

\newcommand{\rooted}{\catsscript{rtd}} 
\newcommand{\trivial}{\catsscript{triv}}
\newcommand{\cxl}{\catsscript{cxl}}

\newcommand{\ps}{\catsscript{ps}}
\newcommand{\pseudo}{\ps}
\newcommand{\str}{\catsscript{str2}}
\newcommand{\strone}{\catsscript{str1}}
\newcommand{\wk}{\catsscript{wk}}

\newcommand{\CCp}{\CC^\mathrm{ps}}
\newcommand{\CCs}{\CC^\str}

\newcommand{\relslice}[2]{#1 \mathbin{\downfibbarrow} #2 }

\newcommand{\compext}[2]{\ensuremath{{#1}.{#2}}}

\newcommand{\reindex}[2]{\ensuremath{#1^*{#2}}}

  \newcommand{\core}[1][]{\ifthenelse{\equal{#1}{}}%
    {\mathrm{core}}%
    {\mathrm{core}(#1)}%
  }
  \newcommand{\supfunctor}[2]{\ifthenelse{\equal{#2}{}}%
    {(-)^\mathrm{#1}}%
    {{#2}^\mathrm{#1}}%
  }

\newcommand{\op}{\operatorname{op}}
\newcommand{\opposite}[1]{\ensuremath{{#1}^{\op}}}

\newcommand{\ob}{\operatorname{ob}}
\newcommand{\mor}{\operatorname{mor}}

\newcommand{\cod}{\operatorname{cod}}

\newcommand{\idmap}{\operatorname{id}}

\newcommand{\sep}{\operatorname{sep}}

\DeclareMathOperator{\im}{im}
\DeclareMathOperator{\repletion}{repl}

\newcommand{\iso}{\cong}
\renewcommand{\equiv}{\simeq}

\newcommand{\defeq}{\coloneq}

\newcommand{\oocomp}{\circ} 

\newcommand{\Ty}{\mathrm{Ty}}
\newcommand{\Tm}{\mathrm{Tm}}

\newcommand{\arrows}[1]{{#1}^\rightarrow}

\newcommand{\fibration}[3]{\begin{tikzcd}[ampersand replacement=\&] {#1} \ar[r, fib, "{#2}"] \& {#3} \end{tikzcd}}
\newcommand{\functor}[3]{\begin{tikzcd}[ampersand replacement=\&] {#1} \ar[r, "{#2}"] \& {#3} \end{tikzcd}}

\ifarxiv %
  \newcommand{\titlerunning}[1]{}
  \newcommand{\keywords}[1]{}
  \newenvironment{credits}{}{}
\fi

\hidenotes

\begin{document}

\title{\titlecaseifaplas{Comparing semantic frameworks for dependently-sorted algebraic theories}}
\titlerunning{\titlecaseifaplas{Comparing semantic frameworks for dependently-sorted algebraic theories}}

\ifaplas
  \author{Benedikt Ahrens\inst{1,4}\orcidID{0000-0002-6786-4538} \and Peter LeFanu Lumsdaine\inst{2}\orcidID{0000-0003-1390-2970} \and Paige Randall North\inst{1,3,4}\orcidID{0000-0001-7876-0956}}
  \institute{Delft University of Technology \email{b.p.ahrens@tudelft.nl} \and Stockholm University \email{p.l.lumsdaine@math.su.se} \and Utrecht University \email{p.r.north@uu.nl} \and University of Birmingham}

\else

  \makeatletter
  \renewcommand{\@maketitle}{%
    \newpage \null
    \vskip 1em%
    \begin{center}%
    \let \footnote \thanks
      {\LARGE \@title \par}%
      \vskip 1em%
      {\begingroup \large \let\and\par \addtolength\parskip{1.5ex}
        \begin{center} \@author \end{center}
        \endgroup}%
    \end{center}%
    \vskip 0em %
  }

  \newcommand{\ORCID}[1]{\href{https://orcid.org/#1}{\textsc{orcid} #1}}
  \newcommand{\email}[1]{\nolinkurl{#1}}
  \newcommand{\authorinfo}[4]{{#1 \\ \small #2 \\ \vspace{-0.7ex} \email{#3} --- \ORCID{#4}}}

  \author{\authorinfo{Benedikt Ahrens}{Delft University of Technology, University of Birmingham}{b.p.ahrens@tudelft.nl}{0000-0002-6786-4538}
    \and \authorinfo{Peter LeFanu Lumsdaine}{Stockholm University}{p.l.lumsdaine@math.su.se}{0000-0003-1390-2970}
    \and \authorinfo{Paige Randall North}{Utrecht University}{p.r.north@uu.nl}{0000-0001-7876-0956}}

  \date{}
\fi

\maketitle

\begin{abstract}
  Algebraic theories with dependency between sorts form the structural core of Martin-Löf type theory and similar systems.
  Their denotational semantics are typically studied using categorical techniques; many different categorical structures have been introduced to model them (contextual categories, categories with families, display map categories, etc.).
  Comparisons of these models are scattered throughout the literature, and a detailed, big-picture analysis of their relationships has been lacking.

  We aim to provide a clear and comprehensive overview of the relationships between as many such models as possible.
  Specifically, we take \emph{comprehension categories} as a unifying language and show how almost all established notions of model embed as sub-2-categories (usually full) of the 2-category of comprehension categories.

\end{abstract}

\ifaplas
\else
\freefootnote{This paper was presented at the Asian Programming Languages Symposium 2024, and appears in its proceedings \cite{ahrens-lumsdaine-north:comparing}. The present version is lightly revised; numbering is unchanged.}
\fi

\autocapsection{Introduction}

Algebraic theories with dependency between their sorts --- that is, the \emph{generalised algebraic theories} of Cartmell \cite{cartmell:generalised-algebraic-theories} and similar frameworks --- are of interest both in their own right and as the structural core of richer type theories such as Martin-Löf type theory \cite{martin-lof:bibliopolis} and its many extensions, Makkai's First Order Logic with Dependent Sorts \cite{makkai:FOLDS}, and others.

The semantics of such systems are usually studied via categorical abstractions.
A veritable zoo of these have been considered: contextual categories \cite{cartmell:thesis}, categories with attributes \cite{cartmell:thesis}, display map categories \cite{taylor:phd}, categories with families \cite{dybjer:internal-type-theory}, type-categories \cite{pitts:categorial-logic}, comprehension categories \cite{jacobs:comprehension-categories}, C-systems \cite{voevodsky:c-system-of-a-module}, B-systems \cite{voevodsky:b-systems}, natural models \cite{awodey:natural-models}, clans \cite{joyal-2017:notes-on-clans}, and more.
Comparisons between many of these have been given in the literature, more are well-known in folklore, and some may be considered too obvious to need spelling out.

However, no accessible overview of this landscape exists.
Here, we aim to give a clear summary of the relationships between these different structures for easy reference at a glance.
What comparison functors connect different kinds of structures?  When are these comparisons equivalences?  And when they are not, how significant is the difference?

\subsubsection*{\titlecaseifaplas{Summary of results}}

We take the 2-category of \emph{comprehension categories} \cite{jacobs:comprehension-categories} and pseudo maps as a unified general setting;
most other models considered in the literature turn out to embed as certain sub-2-categories theoreof.

The bulk of this paper consists of laying out these embeddings, the comparisons between them, and their properties.
The models fall naturally into two groups: first (\cref{sec:types-as-maps}) those where types are represented as certain “display maps”, and second (\cref{sec:types-as-primitive}) those where types are a primitive notion, such as contextual categories and categories with families.

The resulting relationships are summarised in \Cref{fig:types-as-maps-overview-diagram,fig:types-as-primitive-overview-diagram}.
The classes of comprehension categories used are defined in \cref{def:conditions-on-compcats} below; most are to be read as conditions either on the fibration of types (\emph{split}, \emph{discrete}, etc.) or on the comprehension functor (\emph{fully faithful}, \emph{injective on objects}, etc.).

We suppress in these diagrams the question of which notions assume a terminal object; but see \Cref{rem:rootedness}.

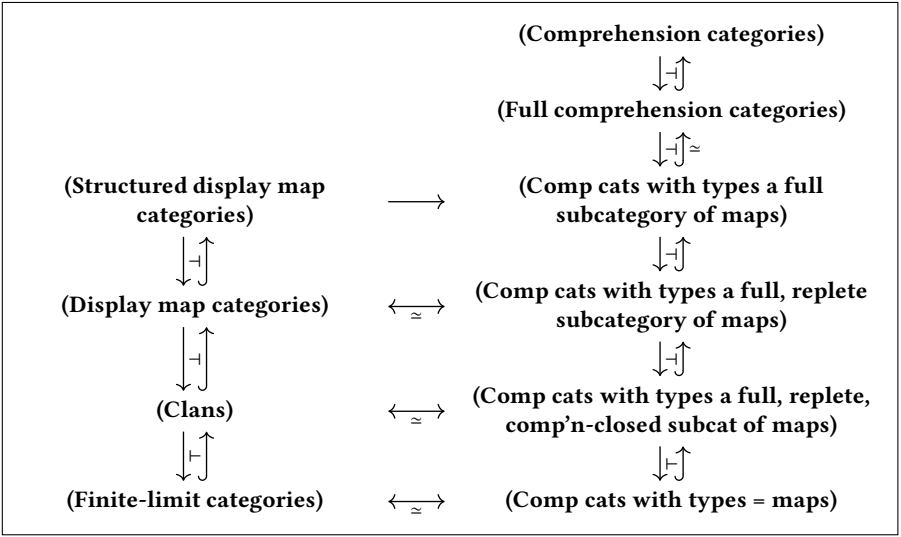
\begin{figure}[!ht]
\begingroup \footnotesize
\newcommand{\boxwide}[1]{\parbox{0.45\textwidth}{\centering \textbf{(#1)}}}
\newcommand{\boxnarrow}[1]{\parbox{0.38\textwidth}{\centering \textbf{(#1)}}}
\begin{tikzcd}[row sep=1.5em,boxedcd]
  & \boxwide{Comprehension categories} \ar[d,ladjd] \\ 
  & \boxwide{Full comprehension categories} \ar[u,inj] \ar[u,adjsymbol,adjshiftl] \ar[d,ladjd]
  \\
  \boxnarrow{Structured display map categories} \ar[r]
  \ar[d,ladjd=1]
  & \boxwide{Comp cats with types a full subcategory of maps} \ar[u,inj,weq'] \ar[u,adjsymbol,adjshiftl] \ar[d,ladjd]
  \\
  \boxnarrow{Display map categories}  \ar[r,<->,weq'] \ar[u,inj] \ar[u,adjsymbol,adjshiftl] \ar[d,ladjd=2]
  & \boxwide{Comp cats with types a full, replete sub\-category of maps} \ar[u,inj] \ar[u,adjsymbol,adjshiftl] \ar[d,ladjd]
  \\
  \boxnarrow{Clans} \ar[r,<->,weq'] \ar[u,inj] \ar[u,adjsymbol,adjshiftl] \ar[d,radjd=-1]
  & \boxwide{Comp cats with types a full, replete, comp'n-closed subcat of maps} \ar[u,inj] \ar[u,adjsymbol,adjshiftl] \ar[d,radjd]
  \\
  \boxnarrow{Finite-limit categories} \ar[r,<->,weq'] \ar[u,inj] \ar[u,adjsymbol,adjshiftr]
  & \boxwide{Comp cats with types = maps} \ar[u,inj] \ar[u,adjsymbol,adjshiftr]
\end{tikzcd} \endgroup
 \caption{Models with types as display maps (\cref{sec:types-as-maps})}
 \label{fig:types-as-primitive-overview-diagram}
\end{figure}

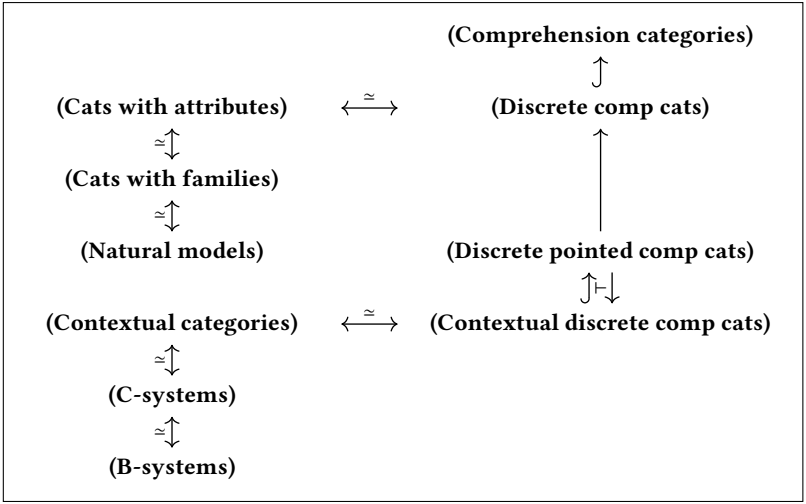
\begin{figure}[!ht] \begingroup \footnotesize
\newcommand{\boxnarrow}[1]{\parbox{0.33\textwidth}{\centering \textbf{(#1)}}}
\newcommand{\boxwide}[1]{\parbox{0.4\textwidth}{\centering \textbf{(#1)}}}
\begin{tikzcd}[row sep=scriptsize,boxedcd]
  & \boxwide{Comprehension categories}
  \\
  \boxnarrow{Cats with attributes}  \ar[r,<->,weq]
  & \boxwide{Discrete comp cats} \ar[u,inj] %
  \\
  \boxnarrow{Cats with families}  \ar[u,<->,weq]
  \\
  \boxnarrow{Natural models}  \ar[u,<->,weq]
  &
  \boxwide{Discrete pointed comp cats} \ar[uu] \ar[d,radjd] \ar[d,adjsymbol,adjshiftl]
  \\
  \boxnarrow{Contextual categories}  \ar[r,<->,weq]
  &
  \boxwide{Contextual discrete comp cats} \ar[u,inj]
  \\
  \boxnarrow{C-systems} \ar[u,<->,weq]
  \\
  \boxnarrow{B-systems} \ar[u,<->,weq]
\end{tikzcd} \endgroup
  \caption{Models with types as primitive (\cref{sec:types-as-primitive})}
  \label{fig:types-as-maps-overview-diagram}
\end{figure}

\subsubsection*{Prerequisites}

Any reader familiar with at least one of the notions of model we survey (categories with families, display map categories, clans, etc.) should be able to follow this paper.
For background on several such models and their motivation for interpreting dependent type theories, we recommend Hofmann \cite{hofmann:syntax-and-semantics} and Jacobs \cite{jacobs:comprehension-categories}.

The 2-categorical language we rely on is minimal — mostly just 2-categories themselves and equivalences and adjunctions between them.
All our 2-categories and functors are strict, and we view 1-categories as locally discrete 2-categories.
For a breezy introduction covering all of these, see Power \cite{power:2-categories}.

\ifarxiv
\fi

\autocapsection{Comprehension categories: a broad church}

Comprehension categories were introduced by Jacobs \cite{jacobs:comprehension-categories} as a common generalisation of earlier models of type dependency.
As he intended, they form a common home in which to compare those notions and others introduced since.
In this section we set up the 2-categories of comprehension categories into which we will later embed the other notions considered, along with key constructions and properties of comprehension categories for later use.

Throughout, we denote isomorphisms by $\iso$, equivalences of (1- and 2-) categories by $\equiv$, and 2-fully-faithful functors (i.e.~inducing equivalences on hom-categories) by $\injto$.

\autocapsubsection{2-categories of comprehension categories}

\begin{definition}\label{def:compcat}
  A \defemph{comprehension category} $\C = (\C,\T,p,\chi)$ consists of
  \begin{enumerate*}[(1)]
  \item a category $\C$ (whose objects we call \defemph{contexts});
  \item a fibration \fibration{\T}{p}{\C} (of \defemph{types}); and
  \item a functor \functor{\T}{\chi}{\arrows{\C}} (\defemph{comprehension}); such that
  \item $\chi$ lies strictly over  $\C$ (that is, $\cod \oocomp \chi = p$), and is cartesian (that is, sends $p$-cartesian maps to pullback squares).
  \end{enumerate*}
   \[
   \begin{tikzcd}[column sep = tiny, row sep = normal, nodes={inner sep=2pt}]
     \T \ar[rr, "\chi"]  \ar[rd, fib, "{p}"'] & &  \arrows{\C}\ar[ld, "\cod"]
     \\
     &
     {\C}
   \end{tikzcd}
 \]

  We write the comprehension $\chi(A)$ of a type $A \in \T_\Gamma$ as $\Gamma.A \fibto A$ (where $\T_\Gamma$ denotes the fiber $p^{-1}\Gamma$). \PLLtodo{consider adding pullback diagram, to show notation and parallel other notions later}
\end{definition}

\begin{definition}\label{def:compat-mor} \leavevmode
\begin{enumerate}
\item A \defemph{pseudo map} $(F,\bar{F},\varphi) : (\C,\T,p,\chi) \to (\C',\T',p',\chi')$ of comprehension categories consists of a functor $F : \C \to \C'$; a functor $\bar{F} : \T \to \T'$ lying (strictly) over $F$, and sending $p$-cartesian maps to $p'$-cartesian maps; and a natural isomorphism $\varphi : \chi' \bar{F} \iso \arrows{F} \chi$ lying (strictly) over the identity natural transformation on $F$ (so $\varphi$ witnesses that $F$ preserves context extension up to isomorphism).
  \begin{equation*}
    \begin{tikzcd}[row sep = tiny, column sep = tiny, inner sep=2pt]
      \T \ar[ddr, fib] \ar[rr, "\chi"] & & \arrows{\C} \ar[ddl] \ar[drrr, "\arrows{F}" near start, crossing over, ""{name=V} near start] \\
     & & & \T' \ar[ddr, fib] \ar[rr, "\chi'"', ""'{name=U}] & & \arrows{\C'} \ar[ddl] \\
     & \C \ar[drrr, "F"] \\
     & & & & \C'
     \ar[from=1-1, to=2-4, "\bar{F}"', crossing over, ""'{name=R} near end]
    \ar[from=R, to=V, Rightarrow, shorten=2em, "\varphi"]
    \end{tikzcd}
    \qquad \quad
   \begin{tikzcd}[row sep=scriptsize,column sep=0]
     F(\Gamma.A) \ar[rr, "\varphi_A"', "\cong" bend left=0] \ar[dr, fib, "F(\chi_{A})"'] 
     & &
     F\Gamma . \bar{F} A %
     \ar[dl, fib, "\chi'_{\bar{F}A}"] 
     \\
     & F\Gamma
   \end{tikzcd}
  \end{equation*}

  \item A \defemph{strict map} is a pseudo map which preserves context extension on the nose; that is, $\chi' \bar{F} = \arrows{F} \chi$, and $\varphi$ is the identity.
\end{enumerate}
\end{definition}

\begin{remark}
  Strict maps of comprehension categories are considered by Blanco \cite{blanco:relating-categorical-approaches}.
  The first source we know for pseudo maps is Curien--Garner--Hofmann \cite[\textsection 5.1]{curien-garner-hofmann:revisiting}\footnote{Note however that their \emph{strict} maps are stronger than ours, strictly preserving chosen cleavings on the fibration of types.}.
  We agree with the latter authors that \emph{maps} of comprehension categories should mean ``pseudo map'' by default; but in the present paper, we will generally explicitly specify whether maps are pseudo or strict.
  Coraglia and Emmenegger \cite[Def.~3.4]{coraglia-emmenegger:2-categorical} further consider \emph{lax} maps, where the natural transformation $\varphi$ of \cref{def:compat-mor} need not be invertible.
\end{remark}

\begin{definition}
A \defemph{transformation} of pseudo maps $(F,\bar{F},\varphi) \tto (G,\bar{G},\gamma)$ consists of a natural transformation $\alpha : F \tto G$, and another $\bar{\alpha} : \bar{F} \tto \bar{G}$ lying (strictly) over $\alpha$, such that for each $A \in \T_\Gamma$ we have $\gamma_A \chi'(\bar{\alpha}_A) = \alpha_{\Gamma . A} \varphi_A$.
\end{definition}

\begin{definition}\label{def:cat_of_compcats}
  We write $\CCp$ (or just $\CC$) for the 2-category of comprehension categories, pseudo maps, and transformations;
  $\CC^{\str}$ for the 2-category of comprehension categories, strict maps, and transformations;
  and $\CC^{\strone}$ for the 1-category of comprehension categories and strict maps.
\end{definition}

\begin{definition}
  Let $(\C,\T,p,\chi)$ be a comprehension category, and $\Gamma \in \C$ any object.
  The \defemph{contextual slice} $\relslice{\C}{\Gamma}$ is the comprehension category in which:
  \begin{enumerate}
  \item objects of $\relslice{\C}{\Gamma}$ are finite sequences $(A_0,\ldots,A_{n-1})$ in which $A_{k} \in \T_{\Gamma.A_0\ldots A_{k-1}}$, for each $0 \leq k < n$;
  \item maps $(A_0,\ldots,A_{n-1}) \to (B_0,\ldots,B_{m-1})$ are maps \[\Gamma.A_0\ldots A_{n-1} \to \Gamma.B_0\ldots B_{m-1}\] in the slice $\C/\Gamma$;
  \item the new fibration of types is the pullback of $\T$ along the functor $\relslice{\C}{\Gamma} \to \C$ sending $(A_0,\ldots,A_{n-1})$ to the context extension $\Gamma.A_0,\ldots,A_{n-1}$;
  \item the new comprehension is given by $(A_0,\ldots,A_{n-1}).B \defeq (A_0,\ldots,A_{n-1},B)$, with the evident projection $\chi(B)$.  
  \end{enumerate}

\noindent
This construction is given for display map categories by Taylor \cite[Def.~8.3.8]{taylor:practical-foundations} and for categories with attributes by Kapulkin and Lumsdaine \cite[Def.~2.4]{kapulkin-lumsdaine:inverse-diagrams}.
\end{definition}

We now delineate various important subclasses of comprehension categories and notation for the resulting full sub-2-categories of $\CompCat$.
\begin{definition}\label{def:conditions-on-compcats}
  A comprehension category $(\C,\T, p, \chi)$ is called:
  \begin{enumerate}
  \item \defemph{full} if $\chi$ is fully faithful (with the 2-category of these denoted $\CC_\full$);
  \item \defemph{subcategorical} if $\chi$ is a full subcategory inclusion, i.e.\ full, faithful, and injective on objects ($\CC_\subcat$);
  \item \defemph{replete} (assuming it is subcategorical) if $\T$ is a replete subcategory of $\arrows{\C}$ ($\CC_{\replete}$);
  \item \defemph{composition-closed} (assuming subcategorical) if $\T$ is closed under composition and includes identities ($\CC_{\subcat,\compclosed})$;
  \item \defemph{trivial} if $\chi$ is an identity ($\CC_\trivial$); \PLLtodo{In hindsight this isn’t a good name: these can be rich and important! And it’s not established. Can we improve it?}
  \item \defemph{discrete} if $p$ is a discrete fibration ($\CC_\discrete$);
  \item \defemph{split} if $p$ is a split fibration ($\CC_\split$).
  \end{enumerate}

  These subscripts combine in the obvious ways: for instance, $\CC_{\full,\split}$ is the 2-category of full comprehension categories where the fibration is split.

  We also sometimes restrict the maps in the split case: we say a map of split comprehension categories is \defemph{split} if it preserves the chosen splitting on the nose, and denote the resulting sub-2-category $\CC_\split^{\split}$.
\end{definition}

\begin{customremark}[\ref*{def:conditions-on-compcats}a]
Many of these conditions are not invariant under equivalence of categories, as they involve on-the-nose equality of objects.
In each case, one can of course generalise them to a property closed under equivalence;
we work with the present versions since they correspond most naturally and tightly to the other structures we are comparing --- display map categories, categories with attributes, and so on.
\PLL{Note: this is the one place where numbering differs slightly from APLAS version — this remark has been moved from after the following definition, replaced with a different remark there.}
\end{customremark}

Lastly we consider some notions which also add restrictions on the maps.
\begin{definition} \label{def:rootedness-etc}
  A \defemph{pointed} comprehension category $\CC$ is one equipped with a distinguished object $\diamond \in \C$; a map is pointed (resp.~strictly pointed) if it preserves $\diamond$ up to specified isomorphism (resp.~on the nose); a transformation is pointed if its value at $\diamond$ commutes with the given isomorphism. We write $\CCp_\pointed$, $\CCs_\pointed$ for the resulting 2-categories.

  A pointed comprehension category is \defemph{rooted} if its distinguished object is terminal (so written $1$) and the functor $\relslice{\C}{1} \to \C$ is essentially surjective (and hence an equivalence) so that every object is isomorphic to some context extension of 1; and 
  \defemph{contextual} if $\relslice{\C}{1} \to \C$ is moreover bijective on objects (and hence an isomorphism), so every object is \emph{uniquely} expressible as an extension of 1.  We write $\CC_\rooted$ and $\CC_\cxl$ for the resulting full sub-2-categories of $\CC_\pointed$.
\end{definition}

\begin{remark} \label{rem:rootedness}
  The question of when to assume rootedness is a flea throughout.
  In the core models of pure dependency --- comprehension categories, display map categories, CwA’s --- it is superfluous: typically, no type-theoretic principle is allowed to distinguish the empty context from others.
  We thus take these notions as unrooted by default (our only significant departure from established definitions).
  Roots arise naturally however as we specialise towards two extremes of the range of models: finite limit categories on the one side, and contextual categories on the other;
  the comparisons of these with more general models therefore involve rooted- or at least pointedness, as for example \Cref{prop:cxl-core-adjoint}.
\end{remark}

\begin{remark} \label{rem:why-2-categories}
    Why do we insist on 2-categories?  1-categories are technically simpler and are used in much of the literature on these structures; e.g., Blanco \cite{blanco:relating-categorical-approaches} compares 1-categories of categories with attributes, contextual categories, and comprehension categories. But general experience from category theory suggests that categorical structures should always be analysed 2-categorically; and, as they should be, these abstract considerations are justified by applications.
    
    Specifically, a comprehensive analysis must include pseudo maps, since maps between non-syntactic comprehension categories are often not strict; and even when they are (typically when the source is contextual or the target replete and so by \Cref{cor:equiv-pseudo-strict-contextual} or \Cref{thm: iso between DMC and comp cat variant} essentially all maps are strict), the strictness of maps may be lost if for instance we pass to strictifications to interpret syntax as in Hofmann \cite{hofmann:lcccs}.

    Once we admit pseudo maps, however, we must also admit 2-cells to keep the resulting (1- or 2-) category well-behaved.  For instance, the syntactic contextual category of a type theory is typically 1-categorically initial in a 1-category of models with strict maps \cite{streicher:semantics-book,boer:initiality}, and bicategorically initial in a 2-category of pseudo maps, but not initial in either sense in a 1-category of pseudo maps.\PLL{Try to recall Vladimir’s old counterexample showing this point too.}
\end{remark}

\autocapsection{Frameworks with types as certain maps}\label{sec:types-as-maps}

In this section, we consider frameworks where types are represented as certain maps of a category — \emph{display maps}.

These first appear in work of the Cambridge group (Hyland, Pitts, and Taylor) from the late eighties \cite{taylor:phd,hyland-pitts:theory-of-constructions} with some variation in details and terminology.
We primarily follow recent literature in our terminology but note historical differences in usage.

We consider three main notions, successively broadening the franchise of display maps by imposing stronger closure conditions:
\begin{enumerate}
\item \emph{display map categories} (and their \emph{structured} variant), assuming just closure under pullback;
\item \emph{clans}, adding closure under composition and identities;
\item \emph{finite-limit categories}, the limiting case in which all maps are display.
\end{enumerate}

\autocapsubsection{Display map categories}\label{sec:dmc}

\begin{definition}[{Hyland--Pitts \cite[\textsection 2.2]{hyland-pitts:theory-of-constructions},  Taylor \cite[Def.~8.3.2]{taylor:practical-foundations}}]\label{def:dmc}
A \emph{display map category}\footnote{These appear in Hyland--Pitts \cite[\textsection 2.2]{hyland-pitts:theory-of-constructions}  as classes satisfying “\textbf{stability}”, and in Taylor \cite[Def.~8.3.2]{taylor:practical-foundations} as \emph{classes of displays}.} is a category $\C$ together with a replete (i.e.\ isomorphism-invariant) subclass $\DM \subseteq \mor (\C)$ of maps (called \emph{display maps} and written $\fibto$), such that display maps pull back along arbitrary maps; that is, for any display map $d$ and map $f$ into its target, there is some display map $f^* d$ that is a pullback of $d$ along $f$:

\begin{diagram}
\label{diagram: pb}
   \begin{tikzcd}[sep=1.5em]
     \cdot \ar[dr,drpb] \ar[dashed,r," "] \ar[d,dashed,fib,"f^* d"'] & \cdot \ar[fib,d,"d"]
     \\ 
     \cdot \ar[r,"f"'] &  \cdot
   \end{tikzcd}
 \end{diagram}
 We call such a $\DM$ a \emph{class of display maps} in $\C$.
\end{definition}

\begin{example}[{\cite[Ex.~8.3.6e]{taylor:practical-foundations}}]
\label{ex: carrable maps}
A map is called \emph{carrable} if it admits pullbacks along arbitrary maps, as in \Cref{diagram: pb}.  Given any class $D$ of carrable maps in a category $\C$, its closure $\DM$ under pullbacks gives a class of display maps in $\C$.
\end{example}

\begin{example}[{Awodey--Warren \cite{awodey-warren}}]
  \label{ex: wfs dmc}
  Important natural examples are given by \emph{weak factorisation systems (WFS's)} and their \emph{algebraic} variants. Given a category with a (possibly algebraic) WFS, we take the display maps to be the right maps of the WFS (often called \emph{fibrations}), which are always stable under pullback. In particular, in any Quillen model category, the fibrations form a class of display maps, yielding as instances Kan fibrations in the category of simplicial sets, or Hurewicz or Serre fibrations in the category of topological spaces.
\end{example}

\begin{lemma}
\label{lem:dmc to comp cat}
  Any display map category $(\C, \DM)$ gives a comprehension category $(\C,\DM,\cod \circ \iota,\iota)$ where $\iota : \DM \injto \arrows{\C}$ is the inclusion of $\D$ viewed as a full subcategory of $\arrows{\C}$.
\end{lemma}

\begin{proof}
   The assumed pullbacks ensure that $\DM \to \C$ is a fibration, and $\iota$ cartesian.
\end{proof}

\begin{definition}
  A map of display map categories is a functor preserving display maps and pullbacks thereof; a transformation of these is simply a natural transformation between functors.
  We write $\DMC$ for the resulting 2-category.
\end{definition}

\begin{theorem}
\label{thm: iso between DMC and comp cat variant}
\Cref{lem:dmc to comp cat} lifts to give an isomorphism and an equivalence
  \[\DMC \iso \CCs_{\replete} \equiv \CC_{\replete}\]
  where these denote the 2-categories of comprehension categories whose comprehension is a replete subcategory inclusion, with strict and pseudo maps respectively (but with all 2-cells in both cases).
\end{theorem}

\begin{proof}
  The assignation of \cref{lem:dmc to comp cat} underlies a 2-functor $\DMC \to \CC$,
  whose image consists of precisely the replete subcategorical comprehension categories.
  It remains to show that it is 2-fully-faithful, and its image on 1-cells consists precisely of the strict maps.
  
  Given display map categories $(\C,\DM)$, $(\C',\DM')$, a map of comprehension categories $(\C,\DM,\cod \circ \iota,\iota) \to (\C',\DM',\cod \circ \iota',\iota')$ amounts to a functor $F: \C \to \C'$ together with a functor $\bar{F} : \D \to \D'$ preserving cartesian morphisms (i.e.\ pullback squares as in \cref{diagram: pb}) together with natural isomorphisms $\varphi_d : \bar{F} d \iso d$.
  Such data $\bar{F}$, $\varphi$ implies (by repleteness of $\DM$) that $F$ preserves display maps and their pullbacks and hence is a map in $\DMC$.
  Conversely, given that $F$ is such a map, suitable $\bar{F}$, $\varphi$ are given by $\arrows{F} |_{\D}$ and the identity isomorphism (yielding a strict map of comprehension categories), and any other such $(\bar{F},\varphi)$ are uniquely isomorphic to these.

  Finally, 2-cells in $\CC$ are pairs $(\alpha,\bar{\alpha}) : (F,\bar{F},\varphi) \to (G,\bar{G},\gamma)$; but since the comprehension of $(\C',\DM')$ is fully faithful, any such $\alpha$ uniquely determines a suitable $\bar{\alpha}$.
\end{proof}

This theorem justifies regarding display map categories precisely as replete subcategorical comprehension categories.

\begin{theorem}
\label{thm: dmc adjunction} \label{thm:fullification} %
    The following inclusions of subcategories of comprehension categories have left adjoints or are equivalences, as shown below.
    \[
    \begin{tikzcd}[column sep=scriptsize]
      \CC_{\replete}
      \ar[r,inj,adjshiftr,weq]
      &
      \CC_{\subcat}
      \ar[l,adjshiftr]
      \ar[r,inj,adjshiftr,weq]
      &
      \CC_{\full}
      \ar[l,adjshiftr]
      \ar[r,inj,adjshiftr]
      &
      \CC
      \ar[l,adjshiftr] \ar[l,adjsymbol]
    \end{tikzcd}
    \]
  \end{theorem}
  
\begin{proof}
  Consider the right-most inclusion $\CC_{\full} \injto \CC$. Its left “fullification” adjoint sends a comprehension category $(\C,\T,p,\chi)$ to $(\C,\T_\chi, p',\chi')$ where $\T \to \T_\chi \to[\chi'] \arrows{\C}$ is the factorisation of $\chi$ as identity-on-objects followed by fully faithful; concretely $\T_\chi$ has the objects of $\T$, but arrows induced by $\chi$ from $\arrows{\C}$.
  Isomorphisms of hom-categories making this a (strict 2-)adjunction follow formally from the fact that the collections of bijective-on-objects functors and fully faithful functors form an orthogonal factorisation system on $\Cat$ \cite[\textsection 4(a)]{lucatelli-nunes-sousa}. Thus we have
    \begin{multline*} 
      \CC_{\full} ((\C_1,(\T_1)_\chi,p_1',\chi_1'), (\C_2,\T_2,p_2,\chi_2))
      \\
      \iso \CC ((\C_1, \T_1, p_1, \chi_1), (\C_2,\T_2,p_2,\chi_2)).
    \end{multline*}

    The middle inclusion $\CC_{\subcat} \injto \CC_{\full}$ has a left adjoint given by factoring a fully faithful comprehension functor $\chi : \T \to \arrows{\C}$ by its image subcategory $\im \chi$; this again gives a strict 2-adjunction, but now moreover a biequivalence, since the unit map $(\C, \T, p, \chi) \to (\C,\im \chi, \iota, \cod)$ is an equivalence in $\CC$.

   Finally, $\CC_{\subcat} \injto \CC_{\replete}$ has a left adjoint sending a subcategory inclusion $\T \injto \arrows{\C}$ to its repletion $\repletion \T \injto \C$.
   Again, the unit maps are equivalences, and we have an isomorphism of hom-categories giving a strict 2-adjunction:
   \begin{multline*} 
      \CC_{\replete} ((\C_1, \repletion \T_1 ,p_1',\chi_1'), (\C_2,\T_2,p_2,\chi_2)) 
      \\
      \equiv \CC_\subcat ((\C_1, \T_1, p_1, \chi_1), (\C_2,\T_2,p_2,\chi_2)).
  \end{multline*}
\end{proof}

\subsubsection{Structured display map categories}\label{sec:sdmc}

The repleteness condition on display maps is occasionally dropped.
The resulting notion is relatively little-used and seems to enjoy few advantages,
perhaps because (as we argue below) their natural maps are “wrong”.

\begin{definition}[{Taylor \cite[Def.~8.3.2]{taylor:practical-foundations}}] \label{def:sdmc}
  A \emph{display structure} on a category $\C$ is a class of maps $\DM \subseteq  \mor (\C)$, again called \emph{display maps}, such that display maps admit all pullbacks as in \Cref{diagram: pb} above.
  A \emph{structured display map category} (sDMC) is a category equipped with a display structure.
\end{definition}

\begin{remark}
  Taylor \cite[Def.~8.3.2]{taylor:practical-foundations} couples repleteness with the question of chosen pullbacks versus existence.
  The latter point matters mainly under a more fine-grained constructive analysis than we aim for; see also \Cref{rem:sdmc-cloven-maps} below.
\end{remark}

\begin{example}
  The category of sets with subset inclusions as display maps forms a sDMC.
\end{example}

The obvious notion of maps is the same as for DMC’s.

\begin{definition}
  Let $\sDMC$ denote the 2-category whose objects are structured display map categories, 1-cells are functors preserving display maps and pullbacks of display maps, and 2-cells are natural transformations.
\end{definition}

Like DMC’s, sDMC’s may be regarded as certain comprehension categories.

\begin{theorem}
\label{thm: iso between sDMC and full comp cats w subcategory inclusion}
    There is an isomorphism
    $ \sDMC \iso \CCs_{\subcat} $
    where the latter 2-category consists of full comprehension categories where $\chi$ is a subcategory inclusion, and with strict maps as 1-cells.
\end{theorem}

In contrast to \cref{thm: iso between DMC and comp cat variant}, pseudo and strict maps do not agree for sDMC’s.

\begin{example}
Take $\C$ to be the full subcategory of $\FinSet$ on $\{0,1,2\}$, with injections as display maps.
With any choice of pullbacks, this gives an sDMC.

Take $\C'$ to be similar but with two isomorphic copies of $1$, so with objects $\{0,1,1',2\}$.
As displays, take all injections, except for maps $1 \to 2$,
where we make the left point $l : 1 \to 2$ and the right point $r' : 1' \to 2$ display maps,
but \emph{not} $r : 1 \to 2$ or $l' : 1' \to 2$.
Pullbacks for a display structure can still be chosen: whenever a pullback yields a point-inclusion into $2$, either $l$ or $r'$ will suffice.

$\C$ and $\C'$ have equivalent repletions, so are equivalent via pseudo maps.
However, no equivalence $F : \C \to \C'$ can strictly preserve display maps,
since $Fl, Fr : F1 \to F2$ would give distinct parallel display maps from either $1$ or $1'$ to $2$.
So not every pseudo map $\C \to \C'$ is isomorphic to a strict one.
\end{example}

\begin{remark} \label{rem:sdmc-cloven-maps}
  If we take sDMC’s to include chosen pullbacks, and additionally require maps of sDMC’s to preserve these on the nose (the definition of \emph{interpretations} in Taylor \cite[Def.~8.3.2]{taylor:practical-foundations} is unclear on this point), then these correspond to maps of \emph{cloven} comprehension categories strictly preserving the cleaving, and so diverge even further from the pseudo maps.
\end{remark}

\begin{theorem}
\label{thm: adj between dmc and sdmc}
    The inclusion $\DMC \injto \sDMC$, or equivalently $\CCs_{\replete} \allowbreak \injto \CCs_{\subcat}$, has a left adjoint.
\end{theorem}
\begin{proof}
  The left adjoint is given by repletion, as in \Cref{thm: dmc adjunction}.\PLL{Consider adding: here it’s a strict 2-adjoint.}
\end{proof}

\subsubsection{Rooted display map categories}
Relatively little changes when we add roots.

\begin{definition}[{\cite[Rem.~8.3.9]{taylor:practical-foundations}}]
    A (possibly structured) display map category is \emph{rooted} if $\C$ has a terminal object, and all morphisms to the terminal object are composites of display maps and isomorphisms. (In the non-structured case, repleteness renders the isomorphisms redundant.) 

    Let $\DMC_\rooted$ be the (non-1-full) sub-2-category of $\DMC$ consisting of rooted display map categories, maps preserving terminal objects, and all transformations.
\end{definition}

\begin{example}
\label{ex: wfs dmc 2}
  Weak factorisation systems, considered as display map categories following \Cref{ex: wfs dmc}, are often rooted:
  for instance, those coming from model categories with all objects fibrant, such as $\Top$ with either Serre or Hurewicz fibrations.
  When this fails, such as in the Kan model structure on simplicial sets, we may still restrict to the full subcategory of fibrant objects (Kan complexes) to recover rootedness.
\end{example}

This notion of rootedness agrees with rootedness for comprehension categories as given in \cref{def:rootedness-etc}:

\begin{theorem} \label{thm:equivs-btn-rooted-dmcs-and-compcats}
  \cref{thm: iso between DMC and comp cat variant,thm: dmc adjunction,thm: iso between sDMC and full comp cats w subcategory inclusion,thm: adj between dmc and sdmc} remain true with rootedness added,
  with one caveat: the “strict” 2-categories should take maps that are strict on comprehension categories, but not necessarily strictly rooted. \PLL{Can’t we set up definitions better to avoid this wart somehow?}
\end{theorem}

We do not restate them in full here: they are summarised in \Cref{fig:summary-types-as-maps-rooted} below.

\autocapsubsection{Clans}
\label{section:clans}

\begin{definition}[{Taylor \cite[\textsection 4.3.2]{taylor:phd}}]\label{def:clan}
  A \emph{clan}\footnote{This name is due to Joyal \cite[Def.~1.1.1]{joyal-2017:notes-on-clans}; in fact these are the original \emph{classes of display maps} of Taylor \cite[\textsection 4.3.2]{taylor:phd}.} is a rooted display map category $(\C, \DM)$ where $\DM$ is closed under composition and contains all identities.
\end{definition}

\begin{example}
  Display map categories arising from WFS's as in \Cref{ex: wfs dmc,ex: wfs dmc 2} are always closed under composition and identities, so are clans whenever they are rooted, i.e.\ when all objects are fibrant.

  We write $\Clan$ for the 2-category of clans, as a full sub-2-category of $\DMC_\rooted$.
\end{example}

\begin{theorem}
\label{thm: clans are equiv to full repl sub comp CCs}
We have
\[\Clan \iso \CCs_{\rooted,\replete,\compclosed} \simeq \CC_{\rooted,\replete,\compclosed}\]
where as in \cref{thm:equivs-btn-rooted-dmcs-and-compcats}, “strict” maps preserve comprehension strictly, but not necessarily the root.
\end{theorem}

\begin{proof}
  Immediate by restriction of the equivalences of \cref{thm:equivs-btn-rooted-dmcs-and-compcats}, \[\DMC_\rooted \iso \CCs_{\rooted,\replete} \equiv \CC_{\rooted,\replete}. \qedhere \] 
\end{proof}

\begin{theorem}\label{thm:left_adjoint_to_compcl_incl}
  The inclusions \[\CC^{\str}_{(\rooted,)\replete,\compclosed} \injto \CC^{\str}_{(\rooted,)\replete}\] have a left adjoint (in four versions: rooted and unrooted, strict and unstrict);
  hence so does the inclusion $\Clan \injto \DMC_\rooted$.
\end{theorem}

\begin{proof}
  We take first the least restrictive case, \[\CC_{\replete,\compclosed} \injto \CC_{\replete} \iso \DMC.\]
  The left adjoint sends a display map category $(\C,\DM)$ to $(\C, \overline{\DM})$, where $\overline{\DM}$ is the closure of $\DM$ under composition.
  It is straightforward to check this gives a (strict 2-)adjoint,
  and does not interact with either rootedness or strictness of maps, so restricts to give the other adjoints desired.
\end{proof}

\autocapsubsection{Finite-limit categories}

\emph{Finite limit categories} (also called \emph{left exact} or \emph{lex} categories) constitute the longest established notion we consider, predating dependent sorts, and with a literature too deep and wide to comprehensively survey.
Logically they model \emph{essentially algebraic theories}, which may be presented syntactically in several ways (see, for instance, \cite{freyd:aspects-of-topoi}, \cite[3.D]{adamek-rosicky}, \cite{palmgren-vickers}) or categorically by \emph{sketches} \cite{kelly:on-the-eat}.
They correspond under Gabriel--Ulmer duality \cite{gabriel-ulmer,adamek-hebert-rosicky:on-eats} to \emph{locally finitely presentable categories}.
Good surveys are given by Adámek and Rosický \cite{adamek-rosicky} and Johnstone \cite[D1--2]{johnstone:elephant}.

\begin{definition}
  We write $\Lex$ for the 2-category of categories with finite limits, functors preserving finite limits, and natural transformations.
\end{definition}

\begin{definition}\label{def:compcat-from-finite-limit-cat}
    A finite-limit category $\C$ determines a clan $(\C, \mor(\C))$, with all morphisms taken as display maps.
\end{definition}

Recall that a comprehension category is called \emph{trivial} if its fibration of types is precisely its codomain fibration.  

\begin{lemma}\label{lem:lex_compcat}
    $ \Lex \iso \CCs_{\rooted,\trivial} \equiv \CC_{\rooted,\trivial}$.
\end{lemma}
\begin{proof}
  The construction of \Cref{def:compcat-from-finite-limit-cat} evidently underlies a 2-functor $\Lex \injto \Clan$;
  this is 2-fully faithful, since preserving pullbacks and the terminal object implies preserving all finite limits.
  Then composing with the isomorphism $\Clan \iso \CCs_{\rooted,\replete,\compclosed}$, the image is precisely $\CC_{\rooted,\trivial}$.
\end{proof}

\begin{theorem} \label{thm:lex-to-clan-radj}
  The inclusion $\Lex \injto \Clan$ (or, equivalently, $\CC_{\rooted,\trivial} \injto \CC_{\rooted,\replete,\compclosed} $) has a right adjoint.
\end{theorem}
\begin{proof}
    The right adjoint sends a clan $(\C,\DM)$ to the full category $\C_{\sep} \subseteq \C$ of objects whose diagonal is a display map (called \emph{separated} objects).

    All maps in $\C_{\sep}$ are display in $\C$ (if $Y$ is separated, any $f : X \to Y$ is the composite of $(f \times Y)^*\Delta_Y : X \fibto X \times Y$ and $\pi_2 : X \times Y \fibto Y$), and finite products and equalisers in $\C_{\sep}$  are direct to construct; so $\C_{\sep}$ is lex and the inclusion $(\C_{\sep},\arrows{\C_{\sep}}) \to (\C,\D)$ is a map of clans. Any other map from a trivial clan to $(\C,\D)$ factors uniquely through $\C_{\sep}$.
  The higher-dimensional parts of the adjunction follow formally.
\end{proof} 

\autocapsection{Frameworks with types as primitive}\label{sec:types-as-primitive}

We turn our attention now to frameworks in which types are not merely certain maps, but a primitive notion.
Compared to the models of \cref{sec:types-as-maps}, those of this section reflect the syntax of type theory more precisely, but are correspondingly further from the natural organisation of more “mathematical” models.

The main group consists of several very closely related notions, essentially reformulations of each other with slightly different emphasis and permitting different generalisations:
categories with attributes \cite{cartmell:thesis,moggi:program-modules},
(split) type-categories \cite{pitts:categorial-logic,garner-van-den-berg:top-and-simp-models}, 
categories with families \cite{dybjer:internal-type-theory},
and natural models \cite{awodey:natural-models}.%
\footnote{As the Swedish saying goes, \emph{kärt barn har många namn}.}
These models may be (and have been) viewed either as \emph{discrete} or as \emph{full split} comprehension categories.

Finally, we reach a venerable and authoritative notion: the \emph{contextual categories} of Cartmell \cite{cartmell:thesis}.
This too has enjoyed several later reformulations as C-systems \cite{voevodsky:c-system-of-a-module} and B-systems \cite{voevodsky:b-systems,ahrens-emmenegger-north-rijke:b-c-systems}.

\autocapsubsection{Categories with families, and equivalents}

We first consider categories with attributes, since they make the comparison with comprehension categories most straightforward.

\begin{definition}[{Cartmell \cite[\textsection 3.2]{cartmell:thesis}\footnote{The original CwA’s of Cartmell \cite[\textsection 3.2]{cartmell:thesis} also included further structure corresponding to type-constructors.  This was stripped down to the present definition by Pitts \cite[Def.~6.9]{pitts:categorial-logic} (there called \emph{type-categories}) and Moggi \cite[Def.~6.2]{moggi:program-modules}, and most subsequent literature has followed suit.}, Moggi \cite[Def.~6.2]{moggi:program-modules}}]
    A \emph{category with attributes (CwA)} consists of a category $\C$; a presheaf $\Ty : \opposite{\C} \to \Set$;
    a functor $(\compext{-}{-}) : \int_\C \Ty \to \C$; and a natural transformation $p : (\compext{-}{-}) \to \pi_1$,
    cartesian in that its naturality squares are pullbacks.
         \[ 
           \begin{tikzcd}[sep=scriptsize]
                \compext{\Gamma'}{\reindex{f}{A}} \ar[r, "f.A"] \ar[d, "p_{\reindex{f}{A}}"'] \ar[dr,drpb]
                &
                \compext{\Gamma}{A} \ar[d, "p_{A}"]
                \\
                \Gamma' \ar[r, "f"]
                & 
                \Gamma                          
           \end{tikzcd}
         \]

  A \emph{(strict) map} of CwA’s is a homomorphism of them considered as essentially algebraic structures in the evident way, or, equivalently, a functor $F : \C \to \C'$ and natural transformation $\bar{F} : \Ty \to \Ty'\cdot F$, commuting on the nose with $(\compext{-}{-})$ and $p$.

  A \emph{pseudo map} consists of $F$ and $\bar{F}$ as in a homomorphism and a natural isomorphism $\varphi : F \compext{-}{-} \iso \compext{F-}{\bar{F}-}$, commuting with $p$ in that $p_FA \varphi_{\Gamma,A} = F p_A$ for all $\Gamma$, $A$.
  A \emph{transformation} of pseudo maps $\alpha : (F,\bar{F},\varphi) \to (G,\bar{G},\psi)$ is a natural transformation $\alpha : F \to G$ such that for each $\Gamma$, $A$, $(\alpha_\Gamma)^*(\bar{G}A) = \bar{F}A$ and $ \alpha_{G\compext{\Gamma}{A}} \psi{\Gamma,A} = \varphi_{\Gamma,A} \compext{\alpha_\Gamma}{\bar{F}A}$. 

  We write $\CwA^{\strone}$ for the 1-category of CwA’s with strict maps, and $\CwA^{\pseudo}$ for their 2-category with pseudo maps and transformations.
\end{definition}

Most literature considers just the 1-category of strict maps; we know no source presenting pseudo maps for CwA’s, though they must be intended in for instance the “suitable 2-category” of \cite[Rem.~2.2.2]{garner-van-den-berg:top-and-simp-models}.
\PLL{Try to recall Voevodsky’s example with showing the category of CwA’s with strict maps was badly behaved!}

It is clear, as noted from the beginning by Jacobs \cite[Ex.~4.10]{jacobs:comprehension-categories}, that categories with attributes simply ``are'' discrete comprehension categories; precisely, we have the following.
\begin{proposition}\label{prop:cwf-ccdisc}
  We have
  \begin{align*}
    \CwA^\strone & \equiv \CompCat_\discrete^\strone, \\
    \CwA^\ps & \equiv \CompCat_\discrete^\ps.
  \end{align*}
\end{proposition}
\begin{proof}
This comes down to the classical equivalence between presheaves and discrete fibrations.
The 1-categorical version is presented in Blanco \cite[Thm.~2.3]{blanco:relating-categorical-approaches}; the 2-categorical version is similarly direct.
\end{proof}

They may be alternatively viewed as full split comprehension categories.

\begin{proposition}
  $\CwA^\strone \equiv \CompCat_{\full,\split}^{\strone,\split}$ and $\CwA^\pseudo \equiv \CompCat_{\full,\split}^{\pseudo,\split}$.
\end{proposition}

Note that even in the pseudo version we restrict to split maps, i.e.\ strictly preserving chosen lifts.

\begin{proof}
  Both equivalences are direct, using fullification in one direction (as in \cref{thm:fullification}) and taking the discrete core of a split fibration in the other.
The 1-categorical equivalence is presented by Blanco \cite[Thm.~2.4]{blanco:relating-categorical-approaches}.
\end{proof}

CwA’s were reformulated by Dybjer to make \emph{terms}, a core component of the syntax of type theory, equally primitive in the semantics.

\begin{definition}[{Dybjer \cite[Def.~1]{dybjer:internal-type-theory}}]
    A \emph{category with families (CwF)} consists of a category $\C$; a presheaf $\Ty$ on $\C$; a presheaf $\Tm$ on $\int_\C \Ty$; and 
    for each $\Gamma \in \C$ and $A \in \Ty(\Gamma)$, an object $\Gamma.A$ and map $p_A : \Gamma . A \to \Gamma$ \emph{representing} $\Tm(A,\Gamma)$ in the sense of a certain universal property.

    A \emph{strict map} of CwF’s consists of a functor and suitable natural transformations on $\Ty$ and $\Tm$, preserving the chosen extensions $\Gamma.A$, $p_A$ on the nose.  These are the only maps considered by Dybjer \cite{dybjer:internal-type-theory} and most literature; we denote their 1-category by $\CwF^\strone$.
    
    A \emph{weak map}\footnote{We would call these pseudo, but it would clash with both \cite[Def.~14]{birkedal-et-al:modal-dtt-and-dras} and \cite[Def.~9]{clairambault-dybjer:biequivalence}.} of CwF’s \cite[Def.~14]{birkedal-et-al:modal-dtt-and-dras} consists of the same data, but preserving context extensions in the weaker sense that their images satisfy the same universal property in the target CwF.  With a suitable notion of transformation, we denote their 2-category $\CwF^\wk$.

    A \emph{pseudo map} of CwF’s \cite[Def.~9]{clairambault-dybjer:biequivalence} is weaker still, preserving reindexing of types and terms only up to coherent isomorphism.  With transformations as defined there, we denote the 2-category of these by $\CwF^\ps$.
\end{definition}

\begin{proposition}\label{prop:cwf-cwa} We have
\begin{enumerate}
  \item $\CwF^\strone \equiv \CwA^\strone$;
  \item $\CwF^\wk \equiv \CwA^\ps$;
  \item $\CwF^\ps \equiv \CompCat_{\full,\split}^{\ps}$.  (Note that we use split comprehension categories here, but do not restrict to split maps.)
\end{enumerate}
\end{proposition}

\begin{proof}
  The core comparison between CwF’s and CwA’s is given in by Hofmann in \cite[\textsection 3.2]{hofmann:syntax-and-semantics} (and formalised by Ahrens, Lumsdaine, and Voevodsky \cite{ahrens-lumsdaine-voevodsky}); checking this extends to the claimed equivalences is routine.
\end{proof}

Natural models \cite{fiore:cwf-repmap-slides,awodey:natural-models} are a further reformulation of categories with families, especially fruitful in paving the way for the massive generalisation by Uemura \cite{uemura:general-framework}.

\begin{definition}
    A \defemph{natural model} consists of a category $\C$, and a pair of objects in $\hat{\C}$ connected by a map $p : \Tm \to \Ty$, which is \defemph{representable} in that the pullback of any representable along it is a representable, and \emph{structured} if it is equipped with a choice of such pullbacks.

    A \defemph{pseudo map} of these is a functor $F : \C \to \C'$ and a commutative square from $p$ to $F^*p'$ in $\hat{\C}$ such that $F$ sends the representing pullbacks of $p$ to representing pullbacks of $F$; a transformation of these is a natural transformation $\alpha : F \to F'$ commuting with the given squares to $F^*p'$, $G^*p'$.  A map of structured natural models is \emph{strict} if it preserves the chosen representations on the nose.

    We write $\NatMod^\ps$ for the 2-category of natural models with pseudo maps and transformations, and $\NatMod^\strone$ for the 1-category of structured natural models and strict maps.
\end{definition}

These maps are defined by Newstead \cite[\textsection 2.3]{newstead:phd} (with the strict as default); given these, it is direct that the comparisons between natural models and CwFs given by Awodey \cite[Prop.~2]{awodey:natural-models} extend to the following equivalences.

\begin{proposition}\label{prop:nat_mod-cwf}
  We have
  \begin{align*}
    \NatMod^\strone &\equiv \CwF^\strone, \\
    \NatMod^\ps &\equiv \CwF^\wk.
  \end{align*}
\end{proposition}

Finally, Van den Berg and Garner \cite[Def.~2.2.1]{garner-van-den-berg:top-and-simp-models} borrow the term ``type-category'' from Pitts \cite[Def.~6.9]{pitts:categorial-logic} but use it for a slightly weaker notion, calling Pitts’ original ones (i.e.~CwA’s) \emph{split} type-categories.
Type-categories in this sense, with the right natural definitions of maps and 2-cells, are straightforwardly shown equivalent to full comprehension categories.

\autocapsubsection{Contextual categories, and equivalents}

Contextual categories are introduced in Cartmell's dissertation \cite[\textsection 2]{cartmell:thesis}.
The definition is rather lengthy; we recall it roughly, and quickly replace it with a much simpler reformulation.

\begin{definition}[{Cartmell \cite[\textsection 2.2]{cartmell:thesis}}]
  A \defemph{contextual category} consists of \begin{enumerate*}[(1)] \item a category $\C$ equipped with a distinguished terminal object 1; \item a tree structure on $\ob{\C}$ with root $1$; \item for each non-root object $A$, a \emph{projection} map $p_A$ from $A$ to its parent; and \item pullbacks of projections along arbitrary maps, which are again projections, $f^* p_A = p_{f^*A}$, strictly functorial in that $\idmap^*A = A$, $(fg)^*A = g^* f^*A$.\end{enumerate*}

  A homomorphism of contextual categories is a functor commuting on the nose with all the given structure.
\end{definition}

The comparison with categories with attributes is direct, and implicit already in Cartmell's work \cite[\textsection 3.2]{cartmell:thesis}.
Call a category with attributes \defemph{contextual} if it is so in the sense of \cref{def:conditions-on-compcats}, when viewed as a discrete comprehension category: that is, each object is uniquely expressible as a context extension of the terminal object.

\begin{proposition}\label{prop:ctxcat-cwa}
  The 1-category $\CxlCat$ of contextual categories and homomorphisms is equivalent to the 1-category $\CwA^\strone_\cxl$ of contextual CwA’s and strict maps, and hence to the 1-category $\CompCat^\strone_{\discrete,\cxl}$ of discrete, contextual comprehension categories and strict maps.
\end{proposition}

The reader may be wondering why for contextual categories, unlike all other notions, we have only introduced strict maps and a 1-category thereof.  This is because in the contextual case, it turns out to make no difference:

\begin{proposition}
  \label{prop: maps out of contextual comp cats}
  Suppose $\C$, $\D$ are comprehension categories, with $\C$ contextual.
  \begin{enumerate}
  \item The inclusion $\CC^{\str}(\C,\D) \injto \CC^\ps(\C,\D)$ is an equivalence.
  \item If $\D$ is discrete, then any transformation $\alpha : F \to G : \C \to \D$ is uniquely determined by its root component $\alpha_1$.
  \item If $\D$ is discrete and terminally pointed, then $\CC^\ps_{\pointed}(\C,\D)$ is essentially discrete, and $\CC^\strone_{\pointed}(\C,\D) \injto \CC^\ps_{\pointed}(\C,\D)$ is an equivalence.
  \end{enumerate}
\end{proposition}

\begin{proof}
  \begin{enumerate*}[(1)]
  \item Any pseudo map may be modified to an isomorphic strict one, by induction on the contextual length of objects of $\C$.
  \item If $\alpha, \beta : F \to G$ agree at $1$, they agree at all objects by induction on length.  
  \item Direct from preceding parts.
  \end{enumerate*}  
\end{proof}

This implies that for contextual categories, unlike CwA’s and similar models, the 1-category of strict maps agrees with the 2-category of pseudo maps, so there is no need to consider pseudo maps or transformations explicitly.
\begin{corollary}
  \newcommand{\sqequiv}{\mathrel{\mkern-0.8mu \equiv \mkern-0.9mu}}
  \label{cor:equiv-pseudo-strict-contextual}
  $\CxlCat \sqequiv \CompCat^\strone_{\discrete,\cxl} \sqequiv \CompCat^\str_{\discrete,\cxl} \sqequiv \CompCat^\ps_{\discrete,\cxl}.$ \qed
\end{corollary}

In this case we drop the superscripts and write just $\CompCat_{\discrete,\cxl}$.  It is straightforward moreover to confirm:

\begin{proposition} \label{prop:cxl-core-adjoint}
  The “contextual core” construction $\C \mapsto \relslice{\C}{\diamond}$ gives right (1- and strict 2-) adjoints to the 2-fully-faithful subcategory inclusions $\CxlCat \injto \CompCat^\strone_{\discrete,\pointed}$, $\CxlCat \injto \CompCat^\ps_{\discrete,\pointed}$.
\end{proposition}

Most presentations of specific contextual categories in the literature — arguably all — can be read as spelling out the contextual core of some simpler, non-contextual category with families.

Later reformulations of contextual categories include the \defemph{C-systems} of Voevodsky \cite[Def.~2.1]{Voevodsky:subsystems} (emphasising them as set-level rather than categorical structures), the \defemph{B-systems} of Voevodsky \cite{voevodsky:b-systems} (an alternative organisation of the dependency between sorts), and the \defemph{$\{w,p,s\}$-GATs} of Garner \cite{garner:combinatorial} (elucidating their combinatorial structure).
In each case, the key parts of an equivalence of 1-categories with $\CxlCat$ are sketched in the cited works introducing them; an equivalence of 1-categories of C- and B-systems is presented explicitly by Ahrens, Emmenegger, North, and Rijke \cite[Thm.~4.1]{ahrens-emmenegger-north-rijke:b-c-systems}.

\autocapsection{Conclusion}\label{sec:conclusion}

\subsubsection*{\titlecaseifaplas{Summary of results}} %

We can now recapitulate the summary diagrams from the introduction more precisely.
\Cref{fig:summary-types-as-maps-unrooted,fig:summary-types-as-maps-rooted} summarise the notions where types are certain maps, in the rooted and unrooted versions respectively; \Cref{fig:summary-types-as-primitive} similarly summarises the notions with primitive types.

As always, $\iso$ and $\equiv$ denote isomorphisms and equivalences respectively, $\injto$ denotes 2-fully-faithful 2-functors (usually literal sub-2-category inclusions),
and super\-script/sub\-script notations on (2-)categories follow \cref{def:conditions-on-compcats} et seq.

\begin{figure}[htbp]
\[
\begin{tikzcd}[row sep=1.8em]
  \CC \ar[d,ladjd, "\dref{thm:fullification}"'] & 
  \CCs \ar[d,ladjd]
  \ar[l,hook']
  \\
  \CC_\full \ar[d,ladjd,"\dref{thm:fullification}"'] \ar[u,inj] \ar[u,adjsymbol,adjshiftl] & 
  \CCs_\full \ar[d,ladjd]
  \ar[l,hook']
  \ar[u,inj] \ar[u,adjsymbol,adjshiftl]
    \\
  \CC_{\subcat}
  \ar[u,inj,weq'] \ar[u,adjsymbol,adjshiftl] \ar[d,ladjd,"\dref{thm:fullification}"']
  & 
  \CCs_{\subcat}
  \ar[u,inj,weq'] \ar[u,adjsymbol,adjshiftl]
  \ar[d,ladjd,"\dref{thm: adj between dmc and sdmc}"']
  \ar[l,hook']
  &
  \sDMC \ar[l,<->,iso', "\dref{thm: iso between sDMC and full comp cats w subcategory inclusion}"]
  \ar[d,ladjd, "\dref{thm: adj between dmc and sdmc}"']
    \\
    \CC_{\replete} \ar[u,inj,weq'] \ar[u,adjsymbol,adjshiftl]
    &
    \CCs_{\replete}
    \ar[l,hook',weq', "\dref{thm: iso between DMC and comp cat variant}"]
    \ar[u,inj] \ar[u,adjsymbol,adjshiftl]
    &
    \DMC  \ar[l,<->,iso', "\dref{thm: iso between DMC and comp cat variant}"]
    \ar[u,inj] \ar[u,adjsymbol,adjshiftl]
  \end{tikzcd}
\]
\caption{Notions with types as maps, unrooted} \label{fig:summary-types-as-maps-unrooted}
\end{figure}

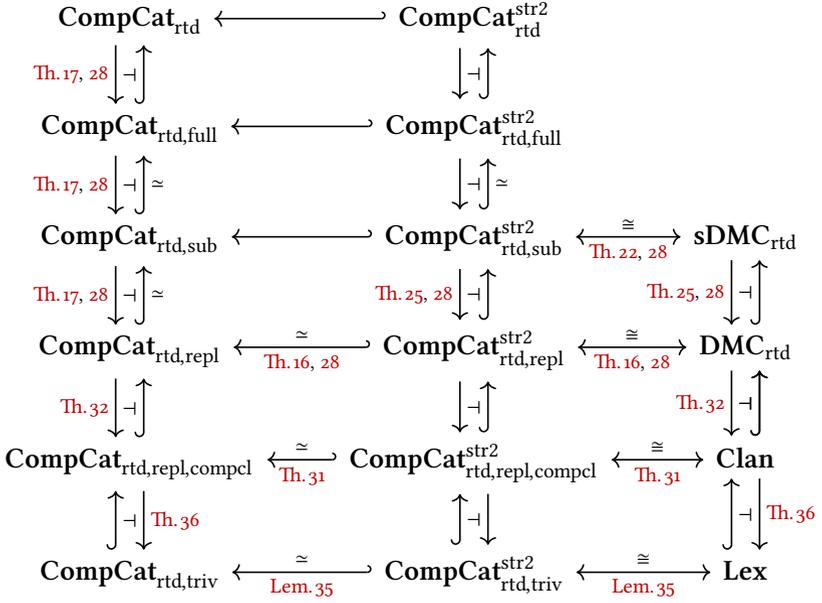
\begin{figure}[!htbp]
\[
\begin{tikzcd}[row sep=1.8em]
    \CC_\rooted \ar[d,ladjd, "\dref{thm:fullification,thm:equivs-btn-rooted-dmcs-and-compcats}"'] & 
    \CCs_\rooted \ar[d,ladjd]
    \ar[l,hook']
  \\ 
    \CC_{\rooted,\full} \ar[d,ladjd, "\dref{thm:fullification,thm:equivs-btn-rooted-dmcs-and-compcats}"'] \ar[u,inj] \ar[u,adjsymbol,adjshiftl] & 
    \CCs_{\rooted,\full} \ar[d,ladjd]
    \ar[l,hook']
    \ar[u,inj] \ar[u,adjsymbol,adjshiftl]
  \\ 
    \CC_{\rooted,\subcat}
    \ar[u,inj,weq'] \ar[u,adjsymbol,adjshiftl] \ar[d,ladjd, "\dref{thm:fullification,thm:equivs-btn-rooted-dmcs-and-compcats}"']
    & 
    \CCs_{\rooted,\subcat}
    \ar[u,inj,weq'] \ar[u,adjsymbol,adjshiftl]
    \ar[d,ladjd, "\dref{thm: adj between dmc and sdmc,thm:equivs-btn-rooted-dmcs-and-compcats}"']
    \ar[l,hook']
    &
    \sDMC_\rooted \ar[l,<->,iso', "\dref{thm: iso between sDMC and full comp cats w subcategory inclusion,thm:equivs-btn-rooted-dmcs-and-compcats}"]
    \ar[d,ladjd, "\dref{thm: adj between dmc and sdmc,thm:equivs-btn-rooted-dmcs-and-compcats}"']
  \\
    \CC_{\rooted,\replete} \ar[u,inj,weq'] \ar[u,adjsymbol,adjshiftl] \ar[d,ladjd, "\dref{thm:left_adjoint_to_compcl_incl}"']
    &
    \CCs_{\rooted,\replete}
    \ar[l,hook',weq', "\dref{thm: iso between DMC and comp cat variant,thm:equivs-btn-rooted-dmcs-and-compcats}"]
    \ar[u,inj] \ar[u,adjsymbol,adjshiftl] \ar[d,ladjd]
    &
    \DMC_\rooted  \ar[l,<->,iso', "\dref{thm: iso between DMC and comp cat variant,thm:equivs-btn-rooted-dmcs-and-compcats}"]
    \ar[u,inj] \ar[u,adjsymbol,adjshiftl] \ar[d,ladjd,"\dref{thm:left_adjoint_to_compcl_incl}"']
  \\
    \CC_{\rooted,\replete,\compclosed} \ar[u,inj] \ar[u,adjsymbol,adjshiftl] \ar[d,radjd, "\dref{thm:lex-to-clan-radj}"]
    & \CCs_{\rooted,\replete,\compclosed}
    \ar[l,hook',weq', "\dref{thm: clans are equiv to full repl sub comp CCs}"] \ar[u,inj] \ar[u,adjsymbol,adjshiftl] \ar[d,radjd]
    & \Clan \ar[u,inj] \ar[u,adjsymbol,adjshiftl] \ar[d,radjd, "\dref{thm:lex-to-clan-radj}"] \ar[l,<->,iso', "\dref{thm: clans are equiv to full repl sub comp CCs}"] \ar[u,inj] \ar[u,adjsymbol,adjshiftl]
  \\
    \CC_{\rooted,\trivial}\ar[u,inj] \ar[u,adjsymbol,adjshiftr]
    & \CCs_{\rooted,\trivial} \ar[l,hook',weq', "\dref{lem:lex_compcat}"] \ar[u,inj] \ar[u,adjsymbol,adjshiftr]
    & \Lex \ar[l,<->,iso', "\dref{lem:lex_compcat}"]\ar[u,inj] \ar[u,adjsymbol,adjshiftr]
  \end{tikzcd}
\]
\caption{Notions with types as maps, rooted} \label{fig:summary-types-as-maps-rooted}
\end{figure}

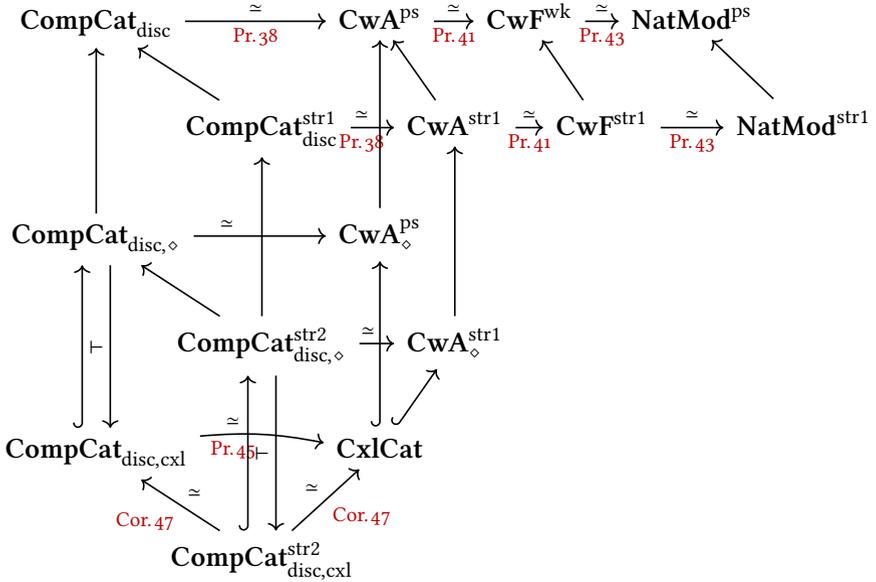
\begin{figure}[!htbp]
\[
\begin{tikzcd}[row sep=1.3em,column sep={4.7em,between origins}]
  \CC_\discrete \ar[rr,<->,weq, "\dref{prop:cwf-ccdisc}"']
  & & \CwA^\ps \ar[rr,<->,weq, "\dref{prop:cwf-cwa}"']
  & &[-3.4em] \CwF^\wk \ar[rr,<->,weq,"\dref{prop:nat_mod-cwf}"']
  & &[-3.0em] \NatMod^\ps &
\\
  & \CC^\strone_\discrete \ar[ul] %
  & & \CwA^\strone \ar[ul] \ar[rr,<->,weq,"\dref{prop:cwf-cwa}"']
  & & \CwF^\strone \ar[ul] \ar[rr,<->,weq,"\dref{prop:nat_mod-cwf}"']
  & & \NatMod^\strone \ar[ul]
\\[0.8em]
  \CC_{\discrete,\pointed} \ar[uu] \ar[rr,<->,weq,pos=0.6]
   \ar[dd,adjsymbol,pos=0.64] \ar[dd,adjshiftl,"\dref{prop:cxl-core-adjoint}" pos=0.65]
  & & \CwA^\ps_\pointed \ar[uu]
  \ar[from=ul,to=ur,<->,weq,pos=0.68,crossing over, "\dref{prop:cwf-ccdisc}"']
\\ 
  & \CCs_{\discrete,\pointed} \ar[uu,crossing over] \ar[ul] %
  & & \CwA^\strone_\pointed \ar[uu] \ar[ul]
\\[0.8em]
  \CC_{\discrete,\cxl} \ar[uu,inj,adjshiftl]
  \ar[rr,<->,bend left=8,start anchor={[xshift=-0.3em,yshift=-0.1ex]},weq,pos=0.65, "\dref{prop:ctxcat-cwa}"' pos=1]
  & & \CxlCat \ar[uu,inj] \ar[ur,inj]
   \ar[from=ul,to=ur,<->,crossing over,weq,pos = 0.6]
\\
  & \CCs_{\discrete,\cxl} \ar[uu,inj,adjshiftl,crossing over]
    \ar[ul,<->,weq',pos=0.4,"\dref{cor:equiv-pseudo-strict-contextual}" pos=0.65]
    \ar[ur,<->,weq,,pos=0.25,"\dref{cor:equiv-pseudo-strict-contextual}"' pos=0.5]
  \ar[from=uu,adjshiftl,crossing over,"\dref{prop:cxl-core-adjoint}" pos=0.19]
  \ar[from=uu,adjsymbol,pos=0.18]
  \end{tikzcd}
\]
\caption{Notions with types primitive} \label{fig:summary-types-as-primitive}
\end{figure}

\subsubsection*{\titlecaseifaplas{Related work}}%

Several other surveys cover ground overlapping with the present paper, some including more detailed analyses of certain models or constructions.

An important early survey is Blanco \cite{blanco:relating-categorical-approaches}, very comparable to the present work but purely 1-categorical and narrower in scope: Blanco relates categories with attributes, contextual categories, and a version of display map categories by embedding them into the (1-)category of comprehension categories and strict maps.

Similarly, Subramaniam \cite[\textsection 1.4]{subramaniam:phd} compares 1-categories of various categorical structures including Lawvere theories and contextual categories.

Ahrens, Lumsdaine, and Voevodsky \cite{ahrens-lumsdaine-voevodsky} compare (split) type categories, categories with families, and relative universe categories, working in univalent type theory and (enabled by this) comparing the \emph{types} of these structures, without considering maps.

Coraglia and Emmenegger \cite{coraglia-emmenegger:2-categorical} construct an equivalence between a 2-cate\-gory of comprehension categories and \emph{lax} maps, and a 2-category of generalized categories with families.
Restricted to discrete comprehension categories and categories with families \cite[Cor.~4.19]{coraglia-emmenegger:2-categorical}, this coincides with our \cref{prop:cwf-ccdisc}.

\subsubsection*{\titlecaseifaplas{Further directions}}%
  For the sake of focus, we have left some variations unexplored here.

  Firstly, we have disregarded questions of whether existence conditions could or should be replaced by chosen structure; for instance, taking fibrations as cloven.
  With the maps we have considered, this is mostly inconsequential;
  but the relationship to maps preserving this structure (e.g.\ cloven maps) could profitably admit more careful analysis.

  Secondly, we have been agnostic about the foundations we work in.
  However, some 1- and strict 2-categorical results in \cref{sec:types-as-primitive} rely on a setting where strict equality of objects in a category is available, e.g., a first-order set-theoretic setting, or using set-based categories in univalent foundations.
  It would be interesting to analyse the relationship between these structures using \emph{univalent categories} in univalent foundations, where equality of objects is not generally strict. 

  Indeed, there is significant interplay between these two issues.
  Working with univalent categories makes many questions of structure vs.\ property moot: essentially unique structure there is genuinely unique, so such structure \emph{is} a mere property.
  For instance, over univalent categories, cloven fibrations coincide with fibrations, chosen pullbacks coincide with their mere existence, and so on;
  so several of the property/structure distinctions ignored here become, in fact, trivial.

\begin{credits}
\subsubsection*{Acknowledgements}
The work presented in this article benefited greatly from meetings and visits supported by COST Action CA20111 EuroProofNet, funded by COST (European Cooperation in Science and Technology), \url{https://www.cost.eu}.
The second author was also supported by the Knut and Alice Wallenberg Foundation project “Type Theory for Mathematics and Computer Science” (PI Thierry Coquand).
This material is based upon work supported by the Air Force Office of Scientific
Research under award number FA9550-21-1-0334.

We are also grateful to many colleagues for fruitful conversations on this topic over the years, including particularly Peter Dybjer, Mike Shulman, André Joyal, Nicola Gambino, and Chaitanya Leena Subramaniam. 
\end{credits}

\ifaplas
  \bibliographystyle{splncs04}
  \bibliography{aplas}
\else
  \bibliographystyle{amsalphaurlmod}
  \bibliography{general-bib}
\fi
  
\end{document}